\documentclass[11pt, a4paper]{amsart}

\usepackage[utf8]{inputenc}
\usepackage{graphicx}
\usepackage{amsmath}
\usepackage{amssymb}
\usepackage{amsbsy}
\usepackage{enumitem}
\usepackage{amsfonts}
\usepackage[english]{babel}
\usepackage{amsthm}
\usepackage[T1]{fontenc}
\usepackage{appendix}
\usepackage{graphicx}
\usepackage{tikz-cd}
\usepackage{cite}
\usepackage{pifont}
\usepackage{dsfont}
\usepackage{multicol}
\usepackage{subcaption}

\usepackage[right=2.5cm, left=2.5cm, top=2.5cm, bottom=2.5cm]{geometry}

\title{\Large{A flower theorem in $\CC^n$}}

\author{Kémo Morvan}
\address{Université Paris Cité, Sorbonne Université, CNRS, IMJ-PRG, F-75013 Paris, France.}
\email{kmorvan@imj-prg.fr}


\theoremstyle{plain}

\newtheorem{mainthm}{Theorem}

\newtheorem{thm}{Theorem}[section]

\newtheorem{lem}[thm]{Lemma}
\newtheorem{prop}[thm]{Proposition}

\theoremstyle{definition}

\newtheorem{definition}[thm]{Definition}
\newtheorem{rem}[thm]{Remark}

\newcommand{\CC}{\mathbb{C}}
\newcommand{\RR}{\mathbb{R}}
\newcommand{\NN}{\mathbb{N}}
\newcommand{\ZZ}{\mathbb{Z}}
\newcommand{\x}{\mathbf{x}}
\newcommand{\y}{\mathbf{y}}
\newcommand{\w}{\mathbf{w}}

\newcommand{\Alp}{\boldsymbol{\alpha}}

\newcommand{\B}{\boldsymbol{B}}

\DeclareMathOperator{\Ree}{Re}
\DeclareMathOperator{\Imm}{Im}
\DeclareMathOperator{\ord}{ord}
\newcommand{\abs}[1]{\left|#1\right|}
\newcommand{\scal}[1]{\langle#1\rangle}
\newcommand{\norm}[1]{\lVert#1\rVert}

\newcommand{\aM}{\langle a,M \rangle}
\newcommand{\jM}{\mathcal{M}}

\begin{document}
	
	\maketitle
	
	\begin{abstract}
		We prove an analog of the flower theorem for non-degenerate reduced tangent to the identity germs that fix the coordinate hyperspaces in any dimension.
	\end{abstract}
	
	\tableofcontents
	
	\section*{Introduction}\label{sect:Intro}
	
	Let $f : (\CC^n, 0) \to (\CC^n, 0)$ be a germ of biholomorphism tangent to the identity, that is, its differential at $0$ is the identity.
	
	In dimension one, the dynamics of such germs are completely described by Leau-Fatou flower theorem \cite{Fatou1920}, \cite{Leau}, which ensures the existence of simply connected domains, called petals, with $0$ in their boundary, covering a punctured neighbourhood of the origin, which are invariant under $f$ or $f^{-1}$ and so that we can conjugate $f$ to a translation on each domain.
	
	There is no equivalent to this result in higher dimensions: one cannot, in general, give such a complete description of the dynamics of $f$ around the origin. However, we can still search for higher dimensional generalizations of the petals, called parabolic manifolds. They can be tangent to complex directions at 0, called characteristic directions. Let $H$ be the homogeneous part of smallest degree in $f - id$; then a direction $[v]$ in $\mathbb{P}^{n-1}$ is said to be characteristic if there exists $\lambda \in \CC$ such that $H(v) = \lambda v$. Such a direction is said to be non-degenerate if $[v]$ is a fixed point for the action of $H$ on $\mathbb{P}^{n-1}$, that is $\lambda \neq 0$, and degenerate if $[v]$ is an indeterminacy point ($\lambda = 0$). In 1998, Hakim \cite{Hakim} showed that for any non-degenerate characteristic direction, there exists a tangent invariant parabolic curve.
	
	In dimension two, the question of finding parabolic manifolds has also been thoroughly investigated, most notably in \cite{EcalleT3}, \cite{Abate01}, \cite{ABT04}, \cite{BMCLH}, \cite{LHRRSS}, \cite{LHRSV}, \cite{Molino}, \cite{Rong15}, \cite{Vivas}. In 2001, Abate showed in \cite{Abate01} that any tangent to the identity map either has a curve of fixed points or admits parabolic curves at the origin.
	
	The proof of this last result relies on Camacho-Sad's index, introduced in \cite{CamachoSad82}, and on a resolution theorem for tangent to the identity maps. This theorem, which is a consequence of a similar result for vector fields due to Seidenberg \cite{Seidenberg}, shows that for any tangent to the identity map $f$, there exists a finite sequence of blow-ups $\pi : (M, E) \to (\CC^2, 0)$ with $E = \pi^{-1} (0)$ such that the lift $\tilde{f} : (M,E) \to (M,E)$ of $f$ fixes $E$ pointwise and satisfies that for every $p \in E$, the germ of $\tilde{f}$ at $p$ is analytically conjugated to one of the following reduced models:
	\[
	\begin{aligned}
		{\rm (i)} \quad \tilde{f} (x,y) &= \left(x+ x^M y^N [1 + A(x,y)], \right. & & \left. y + x^M y^N B(x,y)\right)\\
		{\rm (ii)} \quad \tilde{f} (x,y) &= \left(x+ x^{M+1} y^N [a + A(x,y)], \right. & &\left. y + x^M y^N \left[by + B(x,y)\right]\right)\\
		{\rm (iii)} \quad  \tilde{f} (x,y) &= \left(x+ x^M y^N \left[x + A(x,y)\right],\right. & & \left.y + x^M y^N B(x,y)\right),
	\end{aligned}
	\]
	where $(x,y)$ are local coordinates such that $E$ is either $\{0\}$, or locally equal to either $\{x=0\}$, $\{y=0\}$ or $\{xy=0\}$, and $M$, $N \in \mathbb{N}$, $a$, $b \in \mathbb{R}$ and $A$ and $B \in \mathbb{C}\{x,y\}$ satisfy, in each case:
	\begin{enumerate}[label=(\roman*)]
		\item \emph{Non-singular case}: $(M, N) \notin \{(0, 0), (1,0)\}$, $\ord A \geq 1$ and, if $N \geq 1$ then $B \in (y)$.
		\item \emph{$\star$1 case}: $M \geq 1$, $ab \neq 0$, $a/b \notin \mathbb{Q}_{>0}$,  $\ord A \geq 1$, $\ord B \geq 2$ and, if $N \geq 1$ then $B \in (y)$.
		\item \emph{$\star$2 case}: $M + N \geq 1$, $\ord A,\, \ord B \geq 2$, if $M \geq 1$ then	$A \in (x)$, if $N \geq 1$ then $B \in (y)$, and the function germs $x + A$ and $B$ have no common factors.
	\end{enumerate}
	
	The dynamics of biholomorphisms of the form $(i)$ where described by Abate \cite[Proposition 2.1]{Abate2001} in the case where $M=0$ and by Brochero-Mart\'inez \cite[Proposition 5.3]{BrocheroMartinez} when $N=0$. More recently, dynamics of biholomorphisms of the form (ii) have been studied by L\'opez-Hernanz and Rosas \cite{LHR}, and case (iii) is currently being investigated.
	
	\begin{rem}\label{rem:X}\hfill
		\begin{enumerate}[label=(\Roman*)]
			\item If $E$ contains $\{x=0\}$, then $M\geq 1$; and if $E$ contains $\{y=0\}$, then $N\geq 1$.
			\item If one considers the saturation $X$ of the infinitesimal generator (\cite{Binyamini}) of the germ of $f$ at $p$, then: case (i) corresponds to a non-singular vector field $X$; case (ii) occurs when $X$ has an non-degenerate singularity at $p$; and case (iii) occurs when $X$ has a saddle-node singularity at $p$.
			\item Consider case $(ii)$ under the extra hypothesis that $N\geq 1$. Then $\tilde{f}$ fixes two one-dimensional manifolds ($\{x=0\}$ and $\{y=0\}$).
			
			\item Consider case $(ii)$ under the extra hypothesis that $N= 0 $. In this case, $\tilde{f}$ only fixes $\{x=0\}$ and has a formal invariant curve tangent to $\{y=0\}$. Nevertheless, L\'opez-Hernanz and Rosas exploit the dimension $2$ hypothesis, in order to perform a logarithmic change of coordinates, inspired by Hakim, which allows them to study the local dynamic from the same normal form under the additional assumption that $\{y=0\}$ is an invariant curve. 
			
		\end{enumerate}
	\end{rem}
	
	Resolution of vector fields is also known in dimension 3 (\cite{Panazzolo}, \cite{MP13}), but there is however no equivalent to the resolution theorem for biholomorphisms in general yet.
	
	In this paper, we are interested in the non-degenerate case in arbitrary dimensions: we consider biholomorphisms of the form:
	\begin{equation}\label{eq:f}
		f (x_1, \ldots, x_n) = \Bigg(x_1 \Big(1+ \x^M \big(a_1 + A_1 (x_1, \ldots, x_n)\big)\Big), \ldots, x_n \Big(1+ \x^M \big(a_n + A_n (x_1, \ldots, x_n)\big)\Big)\Bigg)
	\end{equation}
	where $\x^M = x_1^{M_1} \ldots x_n^{M_n}$ with $M_i \in \NN$ for all $i$, $a_i \in \CC^*$ and $A_i : (\CC^n, 0) \to (\CC, 0)$ a holomorphic map for all $i$.
	
	This form generalizes the form (ii) obtained in the resolution theorem in dimension two when $N > 0$, when the exceptional divisor is defined by $\{xy = 0\}$. If we consider the saturation $X$ of the infinitesimal generator of $\tilde{f}$ at the origin, then it has a singularity at $0$ and its linear part is diagonal and invertible. The normal form also implies that $X$ fixes $n$ hyperplanes given by $\{x_i = 0\}$.
	
	Our main result, see Theorem \ref{thm:conv} below, is similar to the one obtained by L\'opez-Hernanz and Rosas in \cite{LHR}: if $f$ satisfies $\Ree \left(\frac{a_i}{\aM} \right) > 0$ for all $i$, we build invariant domains for $f$ or $f^{-1}$ where the forward or backward orbits converge to $0$ and which, together with the fixed set, cover a neighborhood of the origin. 
	We also build Fatou coordinates on these domains, that is we conjugate $f$ with $(z, w_1, \ldots, w_{n-1}) \mapsto (z+1, w_1, \ldots, w_{n-1})$.

	More precisely, set $d = \gcd (M)$, $m_i = \frac{M_i}{d}$ and $\bar{m} = \# \{i, m_i = 0\}$; and define $T_j : \CC^n \to \CC^n$ by $T_j (z_1, \ldots, z_n) = (z_1 - j, z_2, \ldots, z_n)$. Our main result is:
	
	\begin{mainthm}\label{thm:conv}
		Let f be a biholomorphism of form \eqref{eq:f} and such that for all $i \in \{1, \ldots, n\}$, we have $\Ree \left(\frac{a_i}{\aM} \right) > 0$. Then in any neighbourhood of the origin there exists $d$ pairwise disjoint connected open sets $\Omega^{+}_0, \ldots, \Omega^{+}_{d-1}$ with $0 \in \partial \Omega^+_k$ for all $k$ and  $d$ pairwise disjoint connected open sets $\Omega^{-}_0, \ldots, \Omega^{-}_{d-1}$ with $0 \in \partial \Omega^-_k$ for all $k$, such that the following assertions hold:
		\begin{enumerate}
			\item The sets $\Omega^+_k$ are invariant by $f$ and $f^{\circ j} \to 0$ as $j \to +\infty$ compactly on $\Omega^+_k$ for all k, and the sets $\Omega^-_k$ are invariant by $f^{-1}$ and $f^{\circ -j} \to 0$ as $j \to +\infty$ compactly on $\Omega^-_k$ for all k.
			\item The sets $\Omega^{+}_0, \ldots, \Omega^{+}_{d-1}, \Omega^{-}_0, \ldots, \Omega^{-}_{d-1}$ together with the fixed set $\{\x^M = 0\}$ form a neighbourhood of the origin.
			\item For each $k$, there exists biholomorphisms $\phi^+_k : \Omega^+_k \to W^+_k \in \CC^n$ and $\phi^-_k : \Omega^-_k \to W^-_k \in \CC^n$ with the following properties :
			\begin{itemize}
				\item $\phi^+_k$ and $\phi^-_k$ conjugate $f$ to $(z, w_1, \ldots, w_{n-1}) \to (z+1, w_1, \ldots, w_{n-1})$,
				\item the sets $W^+_k$ and $W^-_k$ satisfy:
				$$
				\bigcup_{\pm j \in \NN} T_j \left(W^{\pm}_k \right)= \CC \times (\CC^*)^{n-\bar{m}-1} \times \CC^{\bar{m}}.
				$$
			\end{itemize}
		\end{enumerate}
	\end{mainthm}
	
	Moreover, in the case where $f$ satisfies $\Ree \left(\frac{a_i}{\aM} \right) < 0$ for some $i$, then we show that outside of the fixed set, no orbit stays in a neighborhood of the origin.

	\begin{mainthm}\label{thm:div}
		Let f be a biholomorphism of form \eqref{eq:f} such that there exists $i \in \{1, \ldots, n\}$ such that $\Ree \left(\frac{a_i}{\aM} \right) < 0$.
		Then there exists a neighbourhood $U$ of the origin such that for any $p \in U$ outside of the fixed point set, there exists $j \in \NN$ such that $f^{\circ j} (p) \notin U$ and $j' \in \NN$ such that $f^{\circ -j'} (p) \notin U$.
	\end{mainthm}
	
	Our proof of these results follows the strategy of the proof of L\'opez-Hernanz and Rosas in dimension two, and we adapt their methods to higher dimensions. We first consider the setting of theorem \ref{thm:conv}. Using precise estimates in Leau-Fatou's flower theorem, we first find $d$ connected invariant open sets $U_1, \ldots, U_d$ with $d = \gcd (M)$ on which the orbits of $F$ converge to $0$ and such that any orbit converging to $0$ enters these sets, see section \ref{sect:Para}.
	
	We then show how to conjugate $f$ on $U_\ell$ to $(z, \w) \to (z+1, \w)$ ($\w \in \CC^{n-1}$). To achieve this, we first find $n-1$ invariant functions $\psi_i$ for $f$ that will play the role of $\w$, see Proposition \ref{prop:psi_I}. We then show that the map $\Phi_\ell : \x \mapsto (\frac{1}{\x^M}, \psi_2 (\x), \ldots, \psi_n (\x))$ conjugates $F$ to a map of the form $(z, \w) = (z + 1 + o(1), \w)$, see equation \eqref{eq:hh}. The main difficulty lies in proving that the map $\Phi_\ell$ is a biholomorphism. The proof relies on Rouché's theorem and precise estimates on $\psi_i$, which in higher dimensions are used in an inductive argument, see Lemma \ref{lem:V}.

	The final step of the proof consists in conjugating $(z, \w) \mapsto (z + 1 + o(1), \w)$ to $(z, \w) \mapsto (z + 1, \w)$. We first prove this for $\w$ fixed, see Lemma \ref{lem:beta}, then we paste together the different conjugations obtained, see Proposition \ref{prop:Main} 5.1. Lastly, we show how the results obtained in the previous sections can be extended to larger open sets so that, together with the same sets for $f^{-1}$, they cover a neighborhood of the origin, see \ref{sect:Proof}.
	
	Finally, the proof of theorem \ref{thm:div} is an adaptation in higher dimension of the proof of theorem 2 of \cite{LHR}, see Section \ref{sect:div}.
	
	The paper is organized as follows: in section \ref{sect:Pre} we introduce some tools and notations that are used later in the paper. In section \ref{sect:Para} we study the dynamics in some open sets. In section \ref{sect:Inv} we find invariant maps for $f$ in these sets, and we use them in sections \ref{sect:App} and \ref{sect:Fatou} to build Fatou coordinates. Section \ref{sect:Proof} is devoted to the end of the proof of theorem \ref{thm:conv}, and in section \ref{sect:div} we prove theorem \ref{thm:div}.
	
	\subsubsection*{Acknowledgements}
	The author is deeply grateful to André Belotto da Silva and Matteo Ruggiero for their numerous advices during the realization of this project. The author is partially supported by the project ``Plan d’investissements France 2030", IDEX UP ANR-18-IDEX-0001.

	\section{Preliminaries}\label{sect:Pre}
	
	Let $f : (\CC^n, 0) \to (\CC^n, 0)$ be of form \eqref{eq:f} 
	$$
	f (x_1, \ldots, x_n) = \Bigg(x_1 \Big(1+ \x^M \big(a_1 + A_1 (x_1, \ldots, x_n)\big)\Big), \ldots, x_n \Big(1+ \x^M \big(a_n + A_n (x_1, \ldots, x_n)\big)\Big)\Bigg)
	$$
	where $\x^M = x_1^{M_1} \ldots x_n^{M_n}$ with $M_i \in \NN$ for all $i$, $a_i \in \CC^*$ and $A_i : (\CC^n, 0) \to (\CC, 0)$ a holomorphic map for all $i$.
	
	Applying a linear change of coordinates of the form $(x_1, \ldots, x_n) \to (\alpha_1 x_1, \ldots, \alpha_n x_n)$ such that $\Alp^M = \frac{-1}{\aM}$, $f$ is still of form \eqref{eq:f} but $a_i$ is replaced by $\tilde{a}_i = \frac{-a_i}{\aM}$, which implies that $\scal{\tilde{a}, M} = -1$. From now on we can thus assume $\aM = -1$.
	
	Recall that we set $d = \gcd (M)$, $m_i = \frac{M_i}{d}$ and $\bar{m} = \# \{i, m_i = 0\}$. We can then permute the coordinates so that $m_i = 0$ if and only if $i \geq n-\bar{m} +1$.
	
	Let us remark that for any $\x \in \CC^n$, we have:
	\begin{equation}\label{eq:xm}
		f^m (\x) = \x^m (1 + \scal{a,m} \x^M + o(\x^M))
	\end{equation}
	where $f^m (\x) = \prod_{i=1}^{n} f_i(\x)^{m_i}$.
	
	If we write $y = \x^m$, this equation translates to $f(y) = y(1- \frac{1}{d} y^d + o(y^d))$ , which is a one-dimensional dynamic system. Let us therefore recall Leau-Fatou flower theorem (\cite{Leau}, \cite{Fatou1920}). Let $a$ be a non-zero complex number and $p \in \NN^*$.
	
	\begin{definition}
		For $\epsilon > 0$ and $\theta \in ]0, \pi/2[$ we define :	
		$$
		C (\epsilon, \theta) = \left\{z \in \CC : \abs{z} < \epsilon, \abs{\arg(z)} < \theta \right\}
		$$
		
		and
		\begin{equation*}
			\tilde{C} (\epsilon, \theta) = C (\epsilon, \theta) \cup \left\{z \in \CC, \abs{z- \frac{\epsilon}{2} e^{-i \theta}} < \frac{\epsilon}{2} \right\} \cup \left\{z \in \CC, \abs{z- \frac{\epsilon}{2} e^{i \theta}} < \frac{\epsilon}{2} \right\}.
		\end{equation*}
	\end{definition}

	\begin{figure}[hb]
		\def\svgwidth{0.6\columnwidth}
		\includegraphics[scale=0.3]{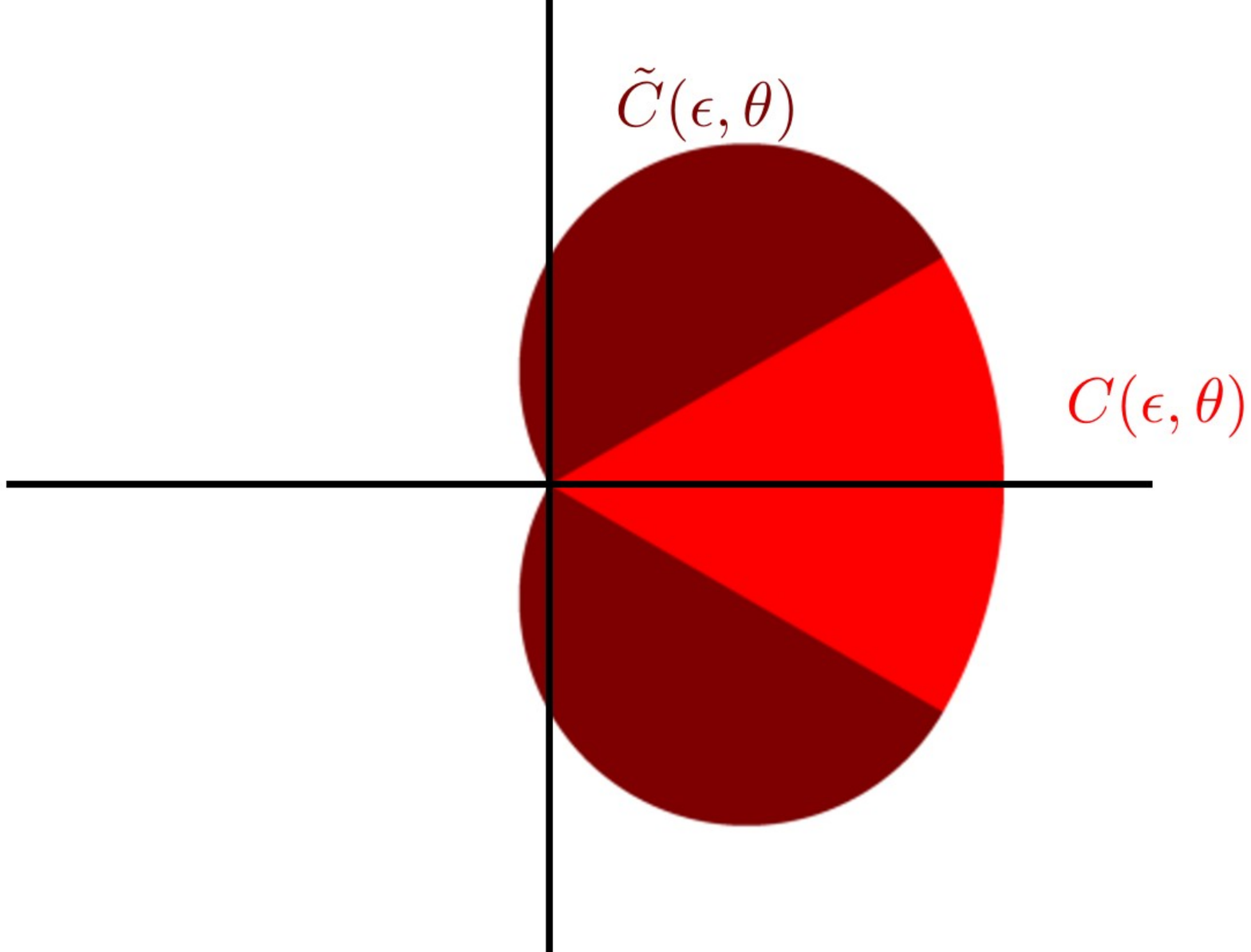}
		\caption{The sets $C$ and $\tilde{C}$}\label{fig:C}
	\end{figure}

	\begin{definition}[see figure \ref{fig:S}]\label{def:S}
		For $\epsilon > 0$ and $\theta \in ]0, \pi/2[$ we define :	
		$$
		S_a (\epsilon, \theta) = \left\{z \in \CC : -a z^p \in C (\epsilon, \theta) \right\}
		$$
		
		and
		\begin{equation*}
			\tilde{S}_a (\epsilon, \theta) = \left\{z \in \CC : -a z^p \in \tilde{C} (\epsilon, \theta) \right\}.
		\end{equation*}
	\end{definition}
	
	\begin{figure}
		\def\svgwidth{0.6\columnwidth}
		\includegraphics[scale=0.35]{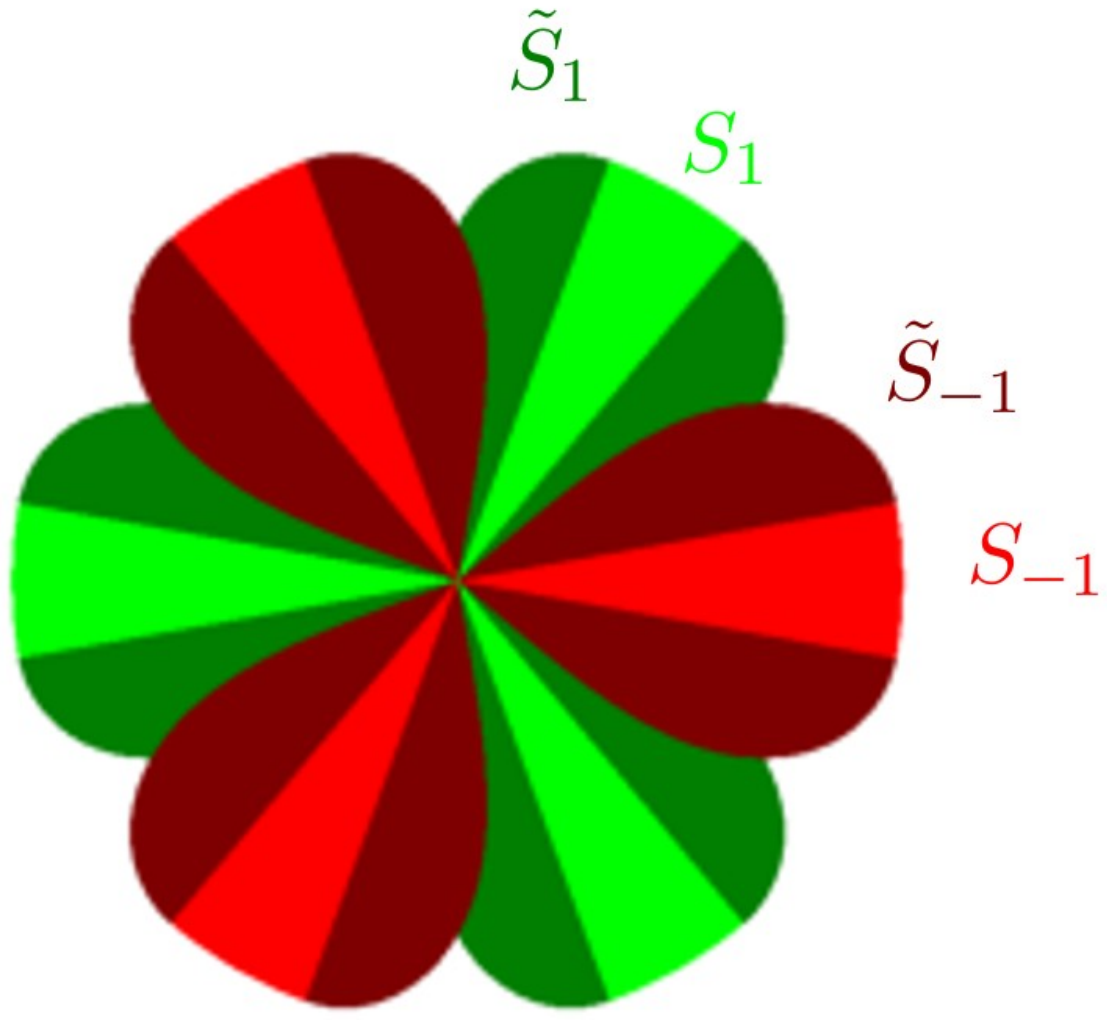}
		\caption{The sets $\tilde{S}_{-1}$ (in red) and $\tilde{S}_{1}$ (in green) for p=3.}\label{fig:S}
	\end{figure}

	\begin{thm}[Leau-Fatou flower theorem]\label{thm:LF}
		Let $f : (\CC, 0) \to (\CC, 0)$ be a biholomorphism of the form: $f (z) = z+ \frac{a}{p} z^{p+1} + o (z^{p+1})$.
		
		Then for all $\theta \in ]0, \pi/2[$, there exists $\epsilon_0 > 0$, $c > 1$ and $C > 0$ such that for every $\epsilon \leq \epsilon_0$ and every component $\tilde{S}$ of $\tilde{S}_a (\epsilon, \theta)$, we have $f (\tilde{S}) \subset \tilde{S}$ and
		$$
		\lim_{j \to \infty} j (f^{\circ j}(z))^p = - a \text{ and } \abs{f^{\circ j} (z)}^p \leq c \frac{\abs{z}^p}{1 + \abs{a} j \abs{z}^p}
		$$
		for all $z \in \tilde{S}$ and $j \in \NN$.
		
		Moreover, if $S$ is the component of $S_a (\epsilon, \theta)$ contained in $\tilde{S}$, then $f (S) \subset S$ and $f^{\circ j} (z) \subset S$ for every $z \in \tilde{S}$ and $j \geq C/\abs{z}^p$.
	\end{thm}
	
	\begin{rem}
		In our setting, $f^m (\x) = \x^m (1 + \scal{a,m} \x^M + o(\x^M))$ but the higher order terms cannot be expressed only with $\x^m$. However, this result still holds as Leau-Fatou's proof does not use the fact that $f$ is holomorphic.
	\end{rem}

	\section{Existence of parabolic domains}\label{sect:Para}
	
	Consider $f : (\CC^n, 0) \to (\CC^n, 0)$ of the form \eqref{eq:f} with $\aM = -1$. The condition in theorem \ref{thm:conv} translates into $\Ree (a_i) < 0$, therefore there exists $\gamma > 0$ such that $\Ree (a_i) + \frac{\gamma}{d} < 0$ for all $i$.
	
	\begin{definition}\label{def:set}
		For $\theta \in ]0, \pi/2[$, $\epsilon$ and $\delta$ positive, we define:
		$$
		D (\epsilon, \theta, \delta) = \{ \x \in \CC^n : \x^M \in C (\epsilon, \theta), \abs{x_i} < \delta \: \forall i\}
		$$
		and 
		$$
		U (\epsilon, \theta) = \{ \x \in \CC^n : \x^M \in C (\epsilon, \theta), \abs{x_i} < \abs{\x^m}^{\gamma} \: \forall i\}.
		$$
		
		We can remark that $U$ is composed of $d$ connected components $U_\ell$, $\ell \in \{0, \ldots, d-1\}$, defined by:
		$$
		U_\ell = \{ \x \in \CC^n : \abs{(\x^m)^d} < \epsilon, \abs{\arg (\x^m) + \frac{2 \pi \ell}{d}} < \frac{\theta}{d}, \abs{x_i} < \abs{\x^m}^{\gamma} \: \forall i\}.
		$$
	\end{definition}

	\begin{figure}[hb]
		\centering
		\begin{subfigure}{0.49\columnwidth}
			\def\svgwidth{1.2\columnwidth}
			\includegraphics[scale=0.3]{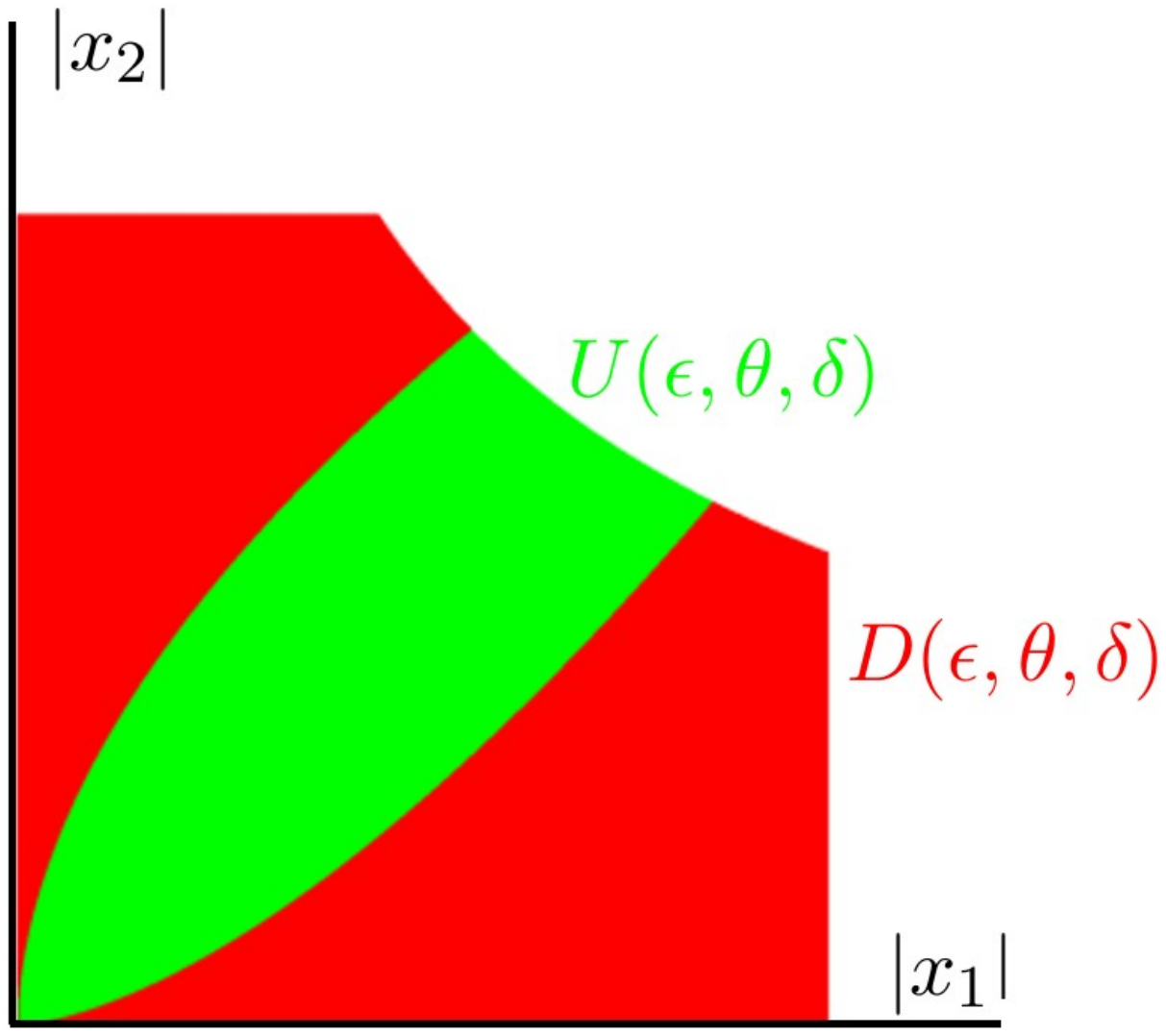}
			\caption{The moduli of the coordinates in $D$ and $U$ \\ in the case $M = (2, 3)$}\label{fig:DU}
		\end{subfigure}
		\begin{subfigure}{0.49\columnwidth}
			\def\svgwidth{1\columnwidth}
			\includegraphics[scale=0.3]{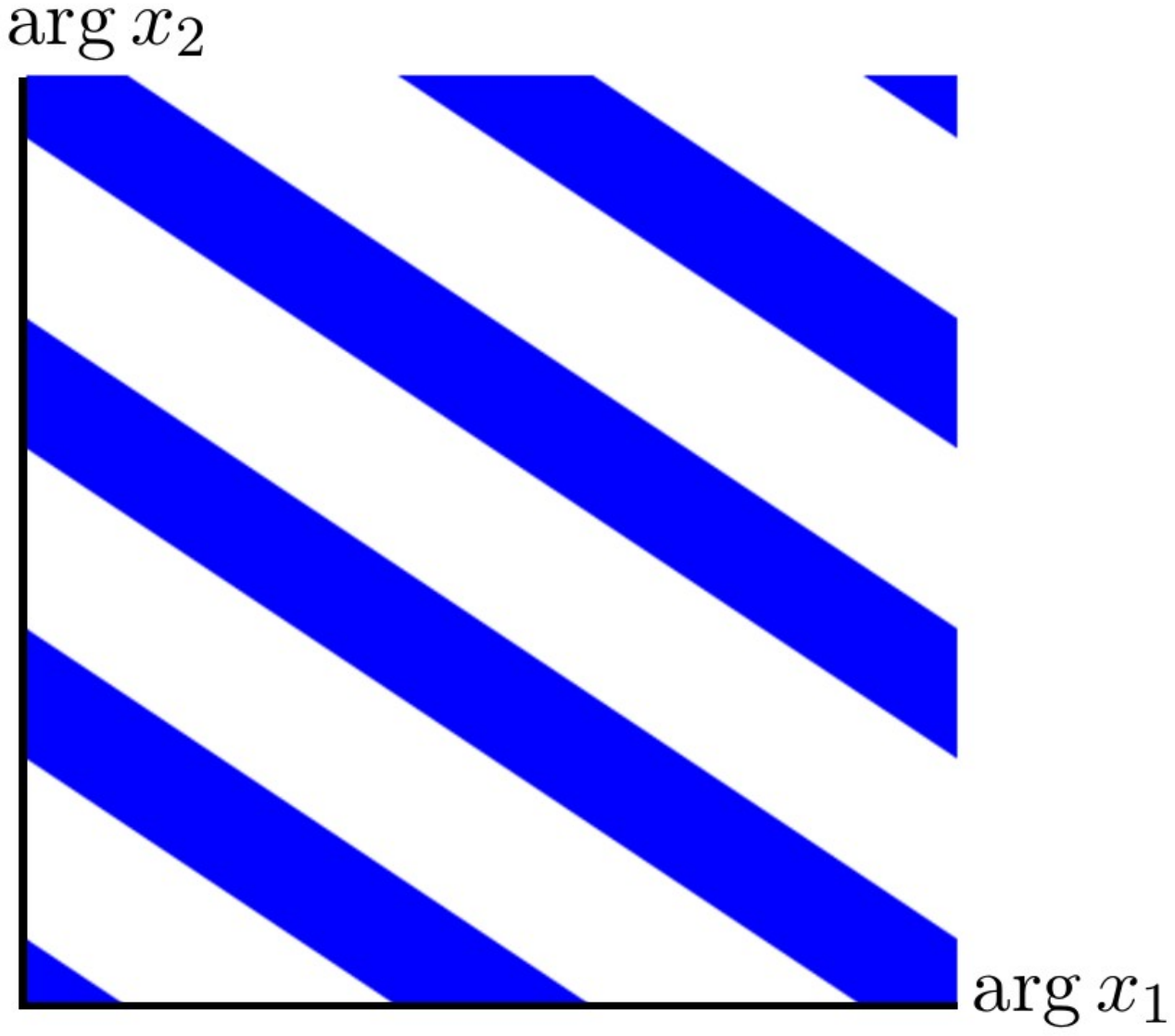}
			\caption{The arguments of the coordinates in $D$ and $U$ in the case $M = (2, 3)$}\label{fig:arg}
		\end{subfigure}
	\end{figure}

	\begin{rem}
		\begin{enumerate}
			\item For all $i \in \{1, \ldots, n - \bar{m}\}$, the set of equations $\abs{x_i} < \abs{\x^m}^{\gamma}$ implies that $x_i \neq 0$, while if $m_i = 0$, the set $U_\ell$ contains elements $\x$ such that $x_i = 0$.
			\item We built these sets in order to exploit equation \eqref{eq:xm}, which is why we have
			$$
			D (\epsilon, \theta, \delta) = \{ \x \in \CC^n : \x^m \in S_{-1} (\epsilon, \theta), \abs{x_i} < \delta \: \forall i\}.
			$$
			In the same way, we will define:
			$$
			\tilde{D} (\epsilon, \theta, \delta) = \{ \x \in \CC^n : \x^m \in \tilde{S}_{-1} (\epsilon, \theta), \abs{x_i} < \delta \: \forall i\}.
			$$
		\end{enumerate}
	\end{rem}

	\begin{prop}\label{prop:cvgce}
		For $\epsilon, \theta > 0$ small enough, we have $f (U (\epsilon, \theta)) \subset U (\epsilon, \theta)$ and for all $\x \in U (\epsilon, \theta)$, $f^{\circ j} (\x) \underset{j\to+\infty}{\longrightarrow} 0$.
	\end{prop}
	
	\begin{proof}
		
		Using equation \eqref{eq:xm} and arguing as in Leau-Fatou theorem, we can find $\epsilon, \theta > 0$ such that for all $\x \in U (\epsilon, \theta, \gamma)$, we have:
		$$
		\abs{(f(\x)^m)^d} < \epsilon \text{ and } \abs{\arg \left((f(\x)^m)^d\right)} < \theta.
		$$
		
		Let $\x \in U (\epsilon, \theta, \gamma)$ and $i \in \{1, \ldots, n\}$, then:
		
		\begin{align*}
			\frac{\abs{f_i (\x)}}{\abs{f(\x)^m}^{\gamma}} &= \frac{\abs{\x_i (1+ \x^M (a_i + o (1)))}}{\abs{\x^m (1 + \scal{a,m} (\x^m)^d + o ((\x^m)^d))}^\gamma}\\
			&= \frac{\abs{\x_i}}{\abs{\x^m}^\gamma} \frac{\abs{1 + a_i \x^M + o(\x^M)}}{\abs{1 - \frac{\gamma}{d} \x^M + o(\x^M)}}\\
			&= \frac{\abs{\x_i}}{\abs{\x^m}^\gamma} \abs{1+ (a_i + \frac{\gamma}{d}) \x^M + o (\x^M)}.
		\end{align*}
		
		Thus if we choose $\epsilon$ and $\theta$ small enough, we get : for all $\x \in U (\epsilon, \theta, \gamma)$,
		\begin{equation}\label{eq:eta}
			\frac{\abs{f_i (\x)}}{\abs{f(\x)^m}^{\gamma}} \leq \frac{\abs{\x_i}}{\abs{\x^m}^\gamma} (1 - \eta \abs{\x^M}) \leq \frac{\abs{\x_i}}{\abs{\x^m}^\gamma}
		\end{equation}
		for some $\eta > 0$. It implies that $F (U) \subseteq U$. Moreover, as $f^{\circ j} (\x) ^m$ converges to $0$ for all $\x$ in $U$, we have that $f^{\circ j}_i (\x)$ also converges to $0$ on $U$.
	\end{proof}
	
	\begin{prop}\label{prop:inU}
		Let $\x \in \CC^n$ such that $f^{\circ j} (\x) \underset{j\to+\infty}{\longrightarrow} 0$. Then there exists $j_0$ such that for all $j \geq j_0$, $f^{\circ j} (\x) \in U$.
	\end{prop}

	\begin{proof}
		By Leau-Fatou theorem, as $f^{\circ j} (\x) ^m$ converges to $0$, we know that there exists $j_1$ such that for all $j \geq j_1$, we have $\abs{(f^{\circ j} (\x)^m)^d} < \epsilon$ and $\abs{\arg ((f^{\circ j} (\x)^m)^d)} < \theta$.
		
		These two conditions are enough to obtain the estimate \eqref{eq:eta}, so we can write:
		\begin{equation}\label{eq:jz}
			\frac{\abs{f^{\circ j}_i (\x)}}{\abs{f^{\circ j}(\x)^m}^{\gamma}} \leq \frac{\abs{f^{\circ j_1}_i (\x)}}{\abs{f^{\circ j_1}(\x)^m}^{\gamma}} \prod_{l = j_1}^{j-1} \left(1 - \eta \abs{f^{\circ l}(\x)^M} \right).
		\end{equation}
		
		Moreover, we know from Leau-Fatou theorem that $\lim_{j \to \infty} j (f^{\circ j}(z))^M = 1$, which implies that $\sum_{l \geq j_1} \abs{(f^{\circ l} (\x)^m)^d)} = + \infty$.
		Therefore the right hand term of \eqref{eq:jz} tends to $0$, so at some point it is smaller than 1 and we have $\x \in U$.
	\end{proof}
	
	\begin{rem}\label{rem:D}
		For $\epsilon, \theta$ and $\delta$ small enough, the set $D = D (\epsilon, \theta, \delta)$ also satisfies $f (D) \subset D$ and for any $\x \in D$, $f^{\circ j} (\x) \to 0$.
		
		As $\Ree (a_i) < 0$ for all $i$, there exists $\rho > 0$ such that for any $\x \in D$ and any $i \in \{1, \ldots, n\}$,
		$$
		\abs{f_i (\x)} = \abs{x_i (1 + \x^m (a_i + o(1)))} < \abs{x_i} (1 - \rho \abs{\x^m}) < \abs{x_i} < \delta,
		$$
		so $D$ is invariant under $f$.
		
		To show the convergence, we use the equality above to obtain:
		$$
		\abs{f_i^{\circ j} (\x)} < \abs{x_i} \prod_{l=0}^{j-1} \left(1 - \rho \abs{f^{\circ l} (\x)^m}\right)
		$$
		and as in the proof of proposition \ref{prop:inU}, this product converges to $0$.
	\end{rem}

	\section{Invariant functions}\label{sect:Inv}
	
	In this section we build holomorphic maps $\psi_I : U_\ell \to \CC$ for $\ell \in \{0, \ldots, d-1\}$  which are invariant under the action of $f$.
	
	We remark that for $\x \in U_\ell$, the argument of $\x^m$ is close to $\frac{2 \pi \ell}{d}$. We can thus define a holomorphic map $U_\ell \to \CC$, $\x \mapsto \log (\x^m)$, which can be used to define $(\x^m)^\lambda = \exp (\lambda \log(\x^m))$ for any $a \in \CC$.
	
	Up to a linear change of coordinates, we can assume $\ell = 0$.
	
	Notice that $f$ is close to the time-1 flow of the vector field
	$$
	X = \x^M (a_1 x_1 \partial_{x_1} + \ldots + a_n x_n \partial_{x_n}).
	$$
	
	\begin{definition}\label{def:g_I}
		Let $I \in \NN^n$ be a n-tuple of integers. We define $g_I$ on $U_{0}$ as :
		$$
		g_I (\x) = \x^I (\x^m)^{\lambda_I}
		$$
		where $\lambda_I = d \scal{a,I}$.
	\end{definition}
	
	Then $g_I$ is invariant under $X$, and thus is almost stable under $f$. Therefore we search invariant functions that are close to these first integrals of $X$.
	
	\begin{prop}\label{prop:psi_I}
		The sequence $\left(g_I (f^{\circ j} (\x))\right)_j$ converges uniformly on $U_0$ to a holomorphic function $\psi_I$ invariant by $f$. Moreover there exists $u_I$ holomorphic such that $\psi_I = g_I u_I$ and $u_I - 1 = O (\abs{\x^m}^{\gamma})$.
	\end{prop}

	\begin{proof}
		We have indeed :
		\begin{align*}
			\frac{g_I (f(\x))}{g_I (\x)} &= \frac{\x^I (\x^m)^{\lambda_I} (1 + \scal{a,I} \x^M + o (\x^M)) \left(1 + \lambda_I \scal{a,m} \x^M + o (\x^M)\right) }{\x^I (\x^m)^{\lambda_I}}\\
			&= 1 + l (\x)
		\end{align*}
		where $l(\x) = \x^M O (\x)$ as $\scal{a,I} + \lambda_I \scal{a,m} = 0$ (recall that $\scal{a,m} = - \tfrac{1}{d}$). Moreover we know that for $\x \in U_0$, we have $f^{\circ j} (\x) \in U_0$ and thus for all $i, j$: $\abs{f^{\circ j}_i (\x)} \leq \abs{f^{\circ j} (\x)^m}^\gamma$. Therefore there exists $K>0$ such that $\abs{l (\x)} \leq K \abs{\x^m}^{d+ \gamma}$ and for all $j$, $\abs{l (f^{\circ j} (\x))} \leq K \abs{f^{\circ j} (\x) ^{m}}^{d + \gamma}$.
		
		Remember that we also get from Leau-Fatou theorem (theorem \ref{thm:LF}) that there exists $c > 0$ such that:
		\begin{equation}\label{eq:fxmd}
			\abs{f^{\circ j} (\x)^m}^{d} \leq c \frac{\abs{\x^m}^d}{1 + j \abs{\x^m}^d}.
		\end{equation}

		The generic term of the product $\prod_{j \geq 0} \frac{g_I (f^{\circ j+1}(\x))}{g_I (f^{\circ j}(\x))}$ is $1 + l (f^{\circ j} (\x))$ and $\abs{l (f^{\circ j} (\x))} \leq \frac{K c}{j^{1 + \gamma/d}}$. Thus the product converges uniformly on $U_0$ to a holomorphic function $u_I$.
		
		Then $g_I (f^{\circ j} (\x)) \to u_I (\x) g_I (\x) = \psi_I (\x)$ uniformly on $U_0$, so $\psi_I$ is well-defined and holomorphic on $U_0$. Moreover, it is clearly invariant by $f$.
		
		For any $\x \in U_0$, we will now show the following three inequalities:
		\begin{equation}
			|u (\x)| \leq 1 + K \abs{\x^{m}}^{\gamma}, \quad |u (\x)| \geq 1 - K \abs{\x^{m}}^{\gamma} \quad \text{and} \quad | \arg (u (\x))| \leq K \abs{\x^{m}}^{\gamma}
		\end{equation}
		for some $K > 0$.
		
		For the first inequality, recall that
		$$
		|u (\x)| = \left| \prod_{j \geq 0} \frac{g_I (f^{\circ j+1}(\x))}{g_I (f^{\circ j}(\x))} \right| = \left| \prod_{j \geq 0} (1 + l \circ f^{\circ j}(\x)) \right|.
		$$
		Using equation \eqref{eq:fxmd}, we obtain
		$$
		|u (\x)| \leq \prod_{j \geq 0} \left(1 + K c \left(\frac{\abs{\x^m}^d}{1 + j \abs{\x^m}^d}\right)^{1 + \tfrac{\gamma}{d}}\right) = \prod_{j \geq 0}\left(1 + K c \left(\frac{1}{j + \frac{1}{\abs{\x^m}^d}}\right)^{1 + \tfrac{\gamma}{d}}\right).
		$$
		After a change of indices, this product is smaller than
		$$
		|u (\x)| \leq \prod_{k \in \tfrac{1}{\abs{\x^m}^d} + \NN} \left(1 + \frac{Kc}{k^{1+ \tfrac{\gamma}{d}}} \right) = \exp \left( \sum_{k \in \tfrac{1}{\abs{\x^m}^d} + \NN} \ln \left(1 + \frac{Kc}{k^{1+ \tfrac{\gamma}{d}}} \right) \right).
		$$
		If $x > 0$, we have $\ln (1 + x) \leq x$ and thus
		$$
		|u (\x)| \leq \exp \left(\sum_{k \in \tfrac{1}{\abs{\x^m}^d} + \NN} \frac{Kc}{k^{1+ \tfrac{\gamma}{d}}} \right).
		$$
		Recall that $\displaystyle \sum_{j \geq j_0} \frac{1}{j^{1 + \alpha}} \sim_{j_0 \to \infty} \frac{1}{\alpha j_0^{\alpha}}$. Therefore,
		$$
		|u (\x)| \leq \exp \left( \frac{2Kcd}{\gamma \left(\tfrac{1}{\abs{\x^m}^d}\right)^{\tfrac{\gamma}{d}}}\right).
		$$
		Finally, we have $\exp (x) \leq 1 + 2 x$ for $x>0$ small enough, and we can conclude
		$$
		|u (\x)| \leq 1 + \frac{4Kcd}{\gamma} \abs{\x^{m}}^{\gamma}.
		$$

		In the same way, we have:
		\begin{align*}
			|u (\x)| &= \left| \prod_{j \geq 0} \frac{g_I (f^{\circ j+1}(\x))}{g_I (f^{\circ j}(\x))} \right| = \left| \prod_{j \geq 0} (1 + l \circ f^{\circ j}(\x)) \right|\\
			& \geq \prod_{j \geq 0} \left(1 -  Kc \left(\frac{\abs{\x^m}^d}{1 + j \abs{\x^m}^d}\right)^{1 + \tfrac{\gamma}{d}}\right) \qquad \text{if } \abs{l \circ f^{\circ j} (\x)} \leq 1 \quad \forall j.
		\end{align*}
		
		Therefore, we can make the same change of indexes:
		$$
		|u (\x)| \geq \prod_{j \geq 0}\left(1 -  Kc \left(\frac{1}{j + \frac{1}{\abs{\x^m}^d}}\right)^{1 + \tfrac{\gamma}{d}}\right) \geq \prod_{k \in \tfrac{1}{\abs{\x^m}^d} + \NN} \left(1 - \frac{Kc}{k^{1+ \tfrac{\gamma}{d}}} \right).
		$$
		We also have $\ln (1 - x) \geq - 2 x$ for $x>0$ small enough, so
		$$
		|u (\x)| \geq \exp \left( \sum_{k \in \tfrac{1}{\abs{\x^m}^d} + \NN} \ln \left(1 - \frac{Kc}{k^{1+ \tfrac{\gamma}{d}}} \right) \right) \geq \exp \left(\sum_{k \in \tfrac{1}{\abs{\x^m}^d} + \NN} \frac{-2Kc}{k^{1+ \tfrac{\gamma}{d}}} \right).
		$$
		As $\displaystyle \sum_{j \geq j_0} \frac{1}{j^{1 + \alpha}} \sim_{j_0 \to \infty} \frac{1}{\alpha j_0^{\alpha}}$,
		$$
		|u (\x)| \leq \exp \left( \frac{-Kcd}{\gamma \left(\tfrac{1}{\abs{\x^m}^d}\right)^{\tfrac{\gamma}{d}}}\right),
		$$
		and since $\exp (-x) \geq 1 - x$ for $x>0$,
		$$
		|u (\x)| \geq 1 - \frac{Kcd}{\gamma} \abs{\x^{m}}^{\gamma}.
		$$
		
		We have thus shown that there exists $\kappa > 0$ such that $1 - \kappa \abs{\x^{m}}^{\gamma} \leq |u (\x)| \leq 1 + \kappa \abs{\x^{m}}^{\gamma}$. In order to show that $u (\x) - 1 = O (\abs{\x^m}^{\gamma})$, we also need to consider the argument of $u (\x)$:
		$$
		| \arg (u (\x))| = \left|\arg \left( \prod_{j \geq 0} \frac{g_I (f^{\circ j+1}(\x))}{g_I (f^{\circ j}(\x))} \right) \right| \leq \sum_{j \geq 0} \left| \arg(1 + l \circ f^{\circ j}(\x)) \right|.
		$$
		For $x \in \CC$ such that $\abs{x} \leq 1/2$, we have $\abs{\arg (1+x)} = \abs{\arctan \frac{\Imm x}{1 + \Ree x}} \leq \abs{\frac{\Imm x}{1 + \Ree x}} \leq \frac{\abs{x}}{1 - \abs{x}} \leq 2 \abs{x}$. Therefore, if $\epsilon$ is small enough so that $\abs{l \circ f^{\circ j}(\x)} \leq 1/2$ for all $\x \in U_0$, we have
		$$
		| \arg (u (\x))| \leq \sum_{j \geq 0} 2 \abs{l \circ f^{\circ j}(\x)} \leq \sum_{j \geq 0}  K c \left(\frac{\abs{\x^m}^d}{1 + j \abs{\x^m}^d}\right)^{1 + \tfrac{\gamma}{d}}
		$$
		and a similar argument gives $| \arg (u (\x))| \leq \kappa \abs{\x^{m}}^{\gamma}$ for some $\kappa > 0$.

	\end{proof}

	\section{Approximate Fatou coordinates}\label{sect:App}
	
	We hence get an invariant holomorphic function for every $I \in \NN^n$. However these functions are not independent : if $I = J_1 + \ldots + J_k$, then we will have $\psi_I = \psi_{J_1} \ldots \psi_{J_k}$. Actually, they all come from applications of the form $\lim_{j \to \infty} (f^{\circ j} (\x))^{\Alp}$ with $\Alp \in \CC^n$ such that $\scal{\Alp, m} = 0$.

	\begin{definition}
		Since $\gcd (m) = 1$, there exists a $n \times n$ invertible matrix $\jM \in M_n (\NN)$ such that the first row of $\jM$ is $m$. Since $m_i = 0$ if $i \geq n - \bar{m} + 1$, we can assume $\jM$ to be of the form
		$$
		\jM = 
		\begin{pmatrix}
			\tilde{M} & 0\\
			0 & I_{\bar{m}}
		\end{pmatrix}
		$$
		with $\tilde{M} \in GL_{n - \bar{m}} (\ZZ)$ with non-negative coefficients.
		
		The inverse $\mathcal{N}$ of $\jM$ is also of the form
		$$
		\mathcal{N} =
		\begin{pmatrix}
			\tilde{N} & 0\\
			0 & I_{\bar{m}}
		\end{pmatrix}
		$$
		with $\tilde{N} \in GL_{n - \bar{m}} (\ZZ)$.
	\end{definition}
	
	\begin{rem}
		For all $i \in \{2, \ldots, n- \bar{m}\}$,  there is a negative entry in the $i$-th column of $\mathcal{N}$. We indeed have $\displaystyle \sum_{j=1}^{n - \bar{m}} m_j \mathcal{N}_{j,i} = 0$, and as $m_j > 0$ for $j \in \{1, n - \bar{m}\}$, we can find a negative number in the $i$-th column.
	\end{rem}
	
	From now on, for any vector $\x \in \CC^n$, we will write $\displaystyle \abs{x}= \sum_{i = 1}^n \abs{x_i}$ and $\displaystyle \norm{x}_\infty = \max_{1 \leq i \leq n} \abs{x_i}$. We use the same notations for matrices.
	
	For $i \geq 2$, let us define $\psi_i = \psi_{\jM_i}$, $g_i = g_{\jM_i}$ and $u_i = u_{\jM_i}$, where $\jM_i$ is the $i$-th row of $\jM$. As $\jM$ is invertible, there is no non-trivial vanishing linear combination of the $\jM_i$ and the $g_i$ maps are independent (that is, the map $\x \mapsto (g_2 (\x), \ldots g_n (\x))$ is surjective).
	
	Remark that as $\psi_i = u_i g_i$ and $u_i$ is close to $1$, we have $\psi_i = 0$ if and only if $g_i = 0$ for $\abs{\x^m}^\gamma$ small enough.

	For $\epsilon$, $\theta$ and $r$ small enough, let us define a domain $V$:
	$$
	V (\epsilon, \theta, r) = \left\{ (z, \w) \in \CC^n : \abs{z} > \epsilon^{-1}, \abs{\arg (z)} < \theta, \abs{\w^{\mathcal{N}_i}} < r \abs{z}^{-\gamma/d - \Ree a_i} \: \forall i \right\}.
	$$
	
	\begin{rem}
		We have that $V \subset \CC \times (\CC^*)^{n-\bar{m}-1} \times \CC^{\bar{m}}$. We indeed know that for any $j \in \{2, \ldots, n - \bar{m}\}$, there exists $i$ such that $\mathcal{N}_{i,j}<0$, and then the equation $\abs{\w^{\mathcal{N}_i}} < r \abs{z}^{-\gamma/d - \Ree a_i}$ implies that $w_j$ is non-zero for all $(z, \w) \in V$.
	\end{rem}
	
	Let $\Phi_\ell : U_\ell \to \CC^n$ be the map defined by:
	$$
	\Phi_\ell (\x) = (1/\x^M, \psi_2 (\x), \ldots, \psi_{n} (\x)).
	$$
	
	\begin{rem}\label{rem:V}
		Let $\x \in \CC^d$ such that the orbit $f^{\circ j} (\x)$ converges to $0$, then we know that $f^{\circ j} (\x) \in U_\ell$ for some $\ell$ and for $j$ big enough. We also have $\Phi_\ell (f^{\circ j} (\x)) \in V$ for $j$ big enough. If we set $(z_j, \w_j) = \Phi_\ell (f^{\circ j} (\x))$ when it is defined (it is well-defined for $j$ big enough such that $f^{\circ j} (\x) \in U_\ell$), then clearly $\abs{z_j} > \epsilon^{-1}$ and $\abs{\arg (z_j)} < \theta$. Moreover as $\w_j$ does not depend on $j$ (as $\psi_i$ is invariant under $f$), the equations $\abs{\w_j^{\mathcal{N}_i}} < r \abs{z_j}^{-\gamma/d - \Ree a_i}$ will be satisfied for $j$ big enough as $\abs{z_j} \to 0$ and $-\gamma/d - \Ree a_i > 0$.
	\end{rem}
	
	\begin{lem}\label{lem:V}
		For $r$ small enough, we have $V \subset \Phi_\ell (U_\ell)$ and $\Phi_\ell : \Phi_\ell^{-1} (V) \to V$ is a biholomorphism.
	\end{lem}
	
	\begin{proof}
		Without loss of generality we can assume $\ell=0$.
		
		We claim that the map $\phi : U_0 \to \CC^n$ defined by :
		$$
		\phi (\x) = \left(\frac{1}{\x^{M}}, g_2 (\x), \ldots, g_{n} (\x) \right)
		$$
		is invertible.
		
		We can indeed write $\phi = \phi_1 \circ \phi_2 \circ \phi_3$ with:
		\begin{itemize}
			\item $\phi_3 (\x) = (\x^m, \x^{\mathcal{M}_2}, \ldots, \x^{\mathcal{M}_{n}}) = \x^{\jM}$,
			\item $\phi_2 (\x) = (\x_1, \x_2 \x_1^{\lambda_{\mathcal{M}_2}}, \ldots, \x_n \x_1^{\lambda_{\mathcal{M}_{n}}})$,
			\item $\phi_1 (\x) = (\frac{1}{\x_1^d}, \x_2, \ldots, \x_n)$.
		\end{itemize}
		Moreover, these three maps are invertible:
		\begin{itemize}
			\item $\phi_3$ is invertible on $U_0$, and its inverse is $\phi_3^{-1} (\x) = \x^{\mathcal{N}}$. We indeed know that $\x \mapsto \x^\jM$ is invertible on $(\CC^*)^n$ for any matrix $M \in GL_n (\ZZ)$, and if $i \in \{1, \ldots, n\}$ is such that $U_0$ intersects $\{x_i = 0\}$, then we have chosen $\jM$ so that $\jM_i = (0, \ldots, 0, 1, 0, \ldots, 0)$ and thus $\x \mapsto \x^{\jM}$ is invertible on $(\CC^*)^{n - \bar{m}} times \CC^{\bar{m}}$, in particular also on $U_0$.
			\item $\phi_2^{-1} (\x) = (\x_1, \x_2 \x_1^{-\lambda_{\mathcal{M}_2}}, \ldots, \x_n \x_1^{-\lambda_{\mathcal{M}_{n}}})$ as $\x_1 \neq 0$,
			\item $\phi_1^{-1} (\x) = (\frac{1}{\x_1^{1/d}}, \x_2, \ldots, \x_n)$. We recall that as $\x \in U_0$, we can define the logarithm of $\x^{m}$ and hence define $(\x^{m})^{a}$ for $a \in \CC$. Moreover, for any $\y \in \phi (U_0)$, only one of the $-d$-roots of $\y_1$ lies in $U_0$ (meaning that there exists $\x \in U_0$ such that $(\x^m)^{-d} = \y_1$). This shows that $\phi_1^{-1}$ is well-defined and holomorphic and the claim is proved.
		\end{itemize}
		We recall that as $\x \in U_0$, we can define the logarithm of $\x^{m}$ and hence define $(\x^{m})^{a}$ for $a \in \CC$. Moreover, for any $\y \in \phi (U_0)$, only one of the $-d$-roots of $\y_1$ lies in $U_0$ (meaning that there exists $\x \in U_0$ such that $(\x^m)^{-d} = \y_1$). This shows that $\phi_1^{-1}$ is well-defined and holomorphic and the claim is proved.
		
		We would like to be more precise when it comes to describing the coefficients of $\phi^{-1} (z, \w)$.
		
		We have $\phi_2^{-1} \circ \phi_1^{-1} (z, \w) = (z^{-1/d}, w_2 z^{\frac{\lambda_{\mathcal{M}_2}}{d}}, \ldots, w_n z^{\frac{\lambda_{\mathcal{M}_n}}{d}})$ with $\lambda_{\mathcal{M}_i} = d \scal{a, \mathcal{M}_i}$ so:
		$$
		\phi_2^{-1} \circ \phi_1^{-1} (z, \w) = (z^{-1/d}, w_2 z^{\scal{a, \mathcal{M}_2}}, \ldots, w_n z^{\scal{a, \mathcal{M}_i}}).
		$$
		
		We then apply $\jM^{-1} = \mathcal{N}$, and we get that for all $i \in \{1, \ldots, n\}$, the $i$-th coordinate of $\phi$ can be written:
		$$
		\phi^{-1}_i (z, \w) = z^{-\mathcal{N}_{i,1}/d} w_2^{\mathcal{N}_{i,2}} z^{\mathcal{N}_{i,2} \scal{a, \mathcal{M}_2}} \ldots w_n^{\mathcal{N}_{i,n}} z^{\mathcal{N}_{i,n} \scal{a, \mathcal{M}_n}}.
		$$
		
		In particular, the exponent of $z$ in $\phi^{-1}_i$ is:
		\begin{align*}
			- \frac{\mathcal{N}_{i,1}}{d} + \sum_{j=2}^n \mathcal{N}_{i,j} \scal{a, \mathcal{M}_j} &= \sum_{j=1}^n \mathcal{N}_{i,j} \scal{a, \mathcal{M}_j}\\
			&= \sum_{j=1}^n \sum_{k=1}^n \mathcal{N}_{i,j} a_k \mathcal{M}_{j,k}\\
			&= \sum_{k=1}^n a_k \sum_{j=1}^n \mathcal{N}_{i,j} \mathcal{M}_{j,k}\\
			&= \sum_{k=1}^n a_k \delta_{i,k} = a_i
		\end{align*}
		as $\scal{a,m} = - \frac{1}{d}$ and $m = \mathcal{M}_1$.
		
		We have thus shown that $\phi^{-1}_i (z, \w) = z^{a_i} \w^{\mathcal{N}_i}$.
		
		We recall that for all $I \in \NN^d$, $\psi_I = u_I g_I$ with $u_I - 1 = O(\abs{\x^m}^\gamma)$.
		
		We can therefore assume that $\epsilon$ is small enough so that for all $i \in \{2, \ldots, n\}$, we have $|u_{\mathcal{M}_i} - 1| < 1/2$. This implies:
		\begin{equation}\label{eq:psimg}
			\abs{\psi_i - g_i} < \tfrac{1}{2} \abs{g_i} \text{ and } \tfrac{1}{2}  < \abs{\frac{\psi_i}{g_i}} < \tfrac{3}{2} .	
		\end{equation}

		Let us now pick $(z,w_2, \ldots, w_n) \in V$. We want to show that this point has exactly one preimage by $\Phi_0$.
		
		In order to do so, we will study maps of the form:
		$$
		G_k (\x) = \left( \frac{1}{\x^M}, g_2 (\x), \ldots, g_k (\x), \psi_{k+1} (\x), \ldots, \psi_n (\x) \right)
		$$
		for $k \in \{1, \ldots, n\}$. We remark that $G_n = \phi$ and $G_1 = \Phi_0$, and we will prove recursively that they are injective.
		
		For $k \in \{1, \ldots, n\}$, let $pr_{\check{k}}: \CC^n \to \CC^{n-1}$ be the projection defined by
		$$
		pr_{\check{k}} (y_1, \ldots, y_n) = (y_1, \ldots, y_{k-1}, y_{k+1}, \ldots, y_n).
		$$
		
		Let us assume that we have shown that $G_{k}$ is injective for some $k \in \{2, \ldots, n\}$, we define a set $A_k$ as:
		$$
		A_k = \left\{ \x \in U_0, pr_{\check{k}} (G_k (\x)) = pr_{\check{k}} (z, \w), \abs{x_j} < \abs{z}^{-\gamma/d} \forall j \in \{1, \ldots, n\} \right\}.
		$$
		Therefore, for all $\x \in A_k$ we have $\x^m = z^{-1/d}$, $g_j (\x) = w_j$ for $j<k$, and $\psi_j (\x) = w_j$ for $j>k$. Notice that $A_k$ is a Riemann surface with boundary as $G_k$ is injective. We indeed have that $g_k : A_k \to \CC$ is injective, which implies that $A_k$ is smooth.
		
		If we show that $\abs{\psi_k - g_k} < \abs{g_k - w_k}$ on the boundary of $A_k$, Rouché's theorem implies that $\psi_k - w_k$ has the same number of $0$ in $A_k$ as $g_k - w_k$, that is exactly one as $G_k$ is injective. Thus $(z,w_2, \ldots, w_n)$ only has one preimage by $G_{k-1}$ as such a preimage needs to be in $A_k$.
		
		We therefore only need the estimate on the boundary of $A_k$.		
		We can decompose $\partial A_k = \bigcup_{i=1}^n \partial_i A_k$ with:
		$$
		\partial_i A_k = \left\{\x \in U_0, pr_{\check{k}} (G_k (\x)) = pr_{\check{k}} (z, \w), \abs{x_j} \leq \abs{z}^{-\gamma/d} \forall j \neq i, \abs{x_i} = \abs{z}^{-\gamma/d} \right\}.
		$$
		
		On $\partial_i A_k$, we thus have:
		\begin{equation}\label{rxi}
			\abs{\w^{\mathcal{N}_i}} \abs{z}^{\Ree a_i} < r \abs{z}^{-\gamma/d} = r \abs{x_i},
		\end{equation}
		where the first inequality comes from the definition of $V$ and the second from the definition of $\partial_i A_k$.	
		Moreover
		$$
		\abs{x_i} = \abs{\phi^{-1}_i (\phi (\x))} = \abs{\phi^{-1}_i (z, w_2, \ldots, w_{k-1}, g_k (\x), \ldots, g_n (\x))}.
		$$
		
		However we also know that for $j>k$, $\psi_k (\x) = w_k$ which implies that $1/2 < \abs{\frac{w_k}{g_k (\x)}} < 2$. We use this inequality in the expression of $\phi^{-1}_i$:
		\begin{align*}
			\abs{\phi^{-1}_i (\phi(\x))} &= \abs{z^{a_i} \cdot \prod_{j=2}^{k-1} w_j^{\mathcal{N}^i_j} \cdot g_k (\x)^{\mathcal{N}^i_k} \cdot \prod_{j=k+1}^n g_j (\x)^{\mathcal{N}^i_j}}\\
			&< \abs{z^{a_i} \cdot \prod_{j=2}^{k-1} w_j^{\mathcal{N}^i_j} \cdot g_k (\x)^{\mathcal{N}^i_k} \cdot \prod_{j=k+1}^n 2^{\abs{\mathcal{N}^i_j}} w_j^{\mathcal{N}^i_j}}	\\
			&= 2^{\abs{\mathcal{N}}} \abs{z^{a_i}} \abs{g_k (\x)}^{\mathcal{N}^i_k}. \prod_{j \neq k}\abs{ w_j}^{\mathcal{N}^i_j}.
		\end{align*}
		
		We apply this result to \eqref{rxi}:
		$$
		\abs{\w^{\mathcal{N}_i}} \abs{z}^{\Ree a_i} < r 2^{\abs{\mathcal{N}}} \abs{z^{a_i}} \abs{g_k (\x)^{\mathcal{N}^i_k} \prod_{j \neq k} w_j^{\mathcal{N}^i_j} },
		$$
		and as $\frac{\abs{z^{a_i}}}{\abs{z}^{\Ree a_i}} = e^{- \Imm a_i \arg z} < e^{\Imm a_i \theta}$, we have:
		
		\begin{equation}\label{eq:wNi}
			\abs{w_k}^{\mathcal{N}^i_k} < r 2^{\abs{\mathcal{N}}} e^{\Imm a_i \theta} \abs{g_k (\x)}^{\mathcal{N}^i_k}.
		\end{equation}

		We now assume $r < 2^{-\norm{\mathcal{N}}_{\infty}} 2^{-\abs{\mathcal{N}}} e^{-\norm{a}_\infty \theta}$, so that $r 2^{\abs{\mathcal{N}}} e^{\Imm a_i \theta} < 2^{-\norm{\mathcal{N}}_{\infty}}$ for all $i$.
		
		There are three cases to consider.
		\begin{enumerate}
			\item If $\mathcal{N}^i_k = 0$, then equation \eqref{eq:wNi} becomes $1 < r 2^{\abs{\mathcal{N}}} e^{\norm{a}_\infty \theta}$ which does not hold for our choice of $r$. As any point in $\partial_i A_k$ should satisfy this equation, we conclude that $\partial_i A_k$ is empty.
			\item If $\mathcal{N}^i_k > 0$, equation \eqref{eq:wNi} becomes $\abs{w_k} < (r 2^{\abs{\mathcal{N}}} e^{\Imm a_i \theta})^{\frac{1}{\mathcal{N}^i_k}} \abs{g_k (\x)}$ which implies $\abs{w_k} < 1/2 \abs{g_k (\x)}$ hence $\abs{g_k (\x) - w_k} \geq \abs{g_k (\x)} - \abs{w_k} > 1/2 \abs{g_k (\x)} > \abs{\psi_k (\x) - g_k (\x)}$ by equation \eqref{eq:psimg}.
			\item Similarly, if $\mathcal{N}^i_k < 0$, then $\abs{w_k} > (r 2^{\abs{\mathcal{N}}} e^{\Imm a_i \theta})^{\frac{1}{\mathcal{N}^i_k}} \abs{g_k (\x)}$ so $\abs{w_k} > 2 \abs{g_k (\x)}$. We thus have $\abs{g_k (\x) - w_k} \geq \abs{w_k} - \abs{g_k (\x)} > 1/2 \abs{g_k (\x)} > \abs{\psi_k (\x) - g_k (\x)}$.
		\end{enumerate}
		
		This concludes the induction. We have thus shown $\Phi_0 : \Phi_0^{-1} (V) \to V$ is injective, and therefore a biholomorphism.
		
	\end{proof}
	
	We can thus use $\Phi_\ell$ to conjugate $f$ to another map $\tilde{F}$ on each open set $U_\ell$. As $\psi_i$ is invariant under $f$ for all $i \in \{2, \ldots, n\}$, $\tilde{F}$ is of the form:
	$$
	\Phi_\ell \circ f \circ \Phi_\ell^{-1} (z, w_2, \ldots, w_n) = (\tilde{f} (z, w_2, \ldots, w_n), w_2, \ldots, w_n).
	$$
	
	If we write $\Phi_\ell^{-1} (z, \w) = \x$, then:
	\begin{equation}\label{eq:hh}
		\begin{aligned}
			\tilde{f} (z, \w) &= \frac{1}{f \circ \Phi_\ell^{-1}(z, \w)^M} = \frac{1}{f(\x)^M} = \frac{1}{\x^M} + 1 + h (\x)\\
			&= z + 1 + \tilde{h}(z, \w)
		\end{aligned}
	\end{equation}
	with $h (\x) = O (\x)$. 
	
	Since $\x \in U_\ell$, we have that $\abs{x_i} < \abs{z}^{-\frac{\gamma}{d}}$, so there exists a constant $K > 0$ such that:
	\begin{equation}\label{eq:mh}
		\abs{\tilde{h} (z, \w)} < K \abs{z}^{-\frac{\gamma}{d}}
	\end{equation}
	for all $(z, \w) \in V$.

	It implies that if $\epsilon$ is small enough, we have $\abs{f(z, \w)} > \abs{z} + \frac{1}{2}$.

	\begin{lem}\label{lem:Vestimates}
		If $\epsilon$, $\theta$ and $r$ are small enough, then $\tilde{F} (V) \subset V$ and there exists a constant $K' > 0$ such that
		$$
		\abs{\frac{\partial \tilde{h}}{\partial z} (z, \w)} < K' \abs{z}^{-1-\frac{\gamma}{d}}
		$$
		for every $(z, \w) \in V$.
	\end{lem}

	\begin{proof}
		Consider a point $(z, \w)$ in $V$. Since $V \subset \Phi_\ell (U_\ell)$, we have that $(\tilde{f} (z, \w), \w) \in \Phi_\ell (U_\ell)$, and thus $\abs{\tilde{f} (z, \w)} > \epsilon^{-1}$ and $\abs{\arg (\tilde{f} (z, \w))} < \theta$. Moreover we know that $\abs{\tilde{f} (z, \w)} > \abs{z}$, so:
		$$
		\abs{\w^{\mathcal{N}_i}} < r \abs{z}^{-\gamma/d - \Ree a_i} < r \abs{\tilde{f} (z, \w)}^{-\gamma/d - \Ree a_i}
		$$
		for all $i \in \{1, \ldots, n\}$. Therefore $\tilde{F} (z, \w) \in V$.
		
		In order to prove the bound on $\tfrac{\partial \tilde{h}}{\partial z}$, we first show that there exists $\rho > 0$ such that for all $(z_0, \w_0) \in V (\epsilon / 2, \theta /2, r/2)$, we have that if $\abs{z - z_0} < \rho \abs{z_0}$, then $(z, \w_0) \in V (\epsilon, \theta, r)$.
		
		Let us consider $(z_0, \w_0) \in V (\epsilon / 2, \theta /2, r/2)$ and $z$ such that $\abs{z - z_0} < \rho \abs{z_0}$. Then:
		$$
		\abs{z} > (1- \rho) \abs{z_0} > (1 - \rho) 2 \epsilon^{-1}.
		$$
		so $\abs{z}>\epsilon^{-1}$ if $\rho$ is small enough. Moreover since $\abs{\frac{z}{z_0} - 1} < \rho$, we have $\abs{\arg (\frac{z}{z_0})} < \arcsin \rho$ so:
		$$
		\abs{\arg z} < \abs{\arg z_0} + \arcsin \rho < \theta/2 + \arcsin \rho
		$$
		so $\abs{\arg z} < \theta$ for $\rho$ small enough. Finally, since $\abs{z_0} < (1 - \rho)^{-1} \abs{z}$, we have for all $i$:
		$$
		\abs{\w_0^{\mathcal{N}_i}} < \frac{r}{2} \abs{z_0}^{-\gamma/d - \Ree a_i} < \frac{r}{2} [ (1 - \rho)^{-1} \abs{z}]^{-\gamma/d - \Ree a_i}
		$$
		so $\abs{\w_0^{\mathcal{N}_i}} < r \abs{z}^{-\gamma/d - \Ree a_i}$ for $\rho$ small enough as $-\gamma/d - \Ree a_i > 0$. Hence $(z, \w_0) \in V (\epsilon, \theta, r)$ if $\rho$ is small enough.
		
		Let $(z_0, \w_0) \in V (\epsilon / 2, \theta /2, r/2)$, then we have shown that $D (z_0, \rho \abs{z_0}) \times \{\w_0\} \subset V (\epsilon, \theta, r)$, so the map
		$$
		\tilde{h}_{\w_0} : z \in D (z_0, \rho \abs{z_0}) \mapsto \tilde{h} (z, \w_0)
		$$
		is well defined and using equation \eqref{eq:mh} we obtain:
		$$
		\abs{\tilde{h}_{\w_0}} < K \abs{z}^{-\gamma/d} < K (1 - \rho)^{-\gamma/d} \abs{z_0}^{-\gamma/d}.
		$$
		Thanks to Cauchy's estimate we then have:
		$$
		\abs{\frac{\partial \tilde{h}}{\partial z} (z_0, \w_0)} = \abs{\tilde{h}_{\w_0}' (z_0)} <  K (1 - \rho)^{-\gamma/d} \abs{z_0}^{-\gamma/d} (\rho \abs{z_0})^{-1} < K' \abs{z_0}^{-1-\gamma/d}.
		$$
		
	\end{proof}
	
	Following the previous Lemmas, if $\epsilon$, $\theta$ and $r$ be small enough and we denote by $V = V (\epsilon, \theta, r)$, we have that:
	\begin{equation}\label{eq:TildeF}
		\begin{matrix}
			\tilde{F} : & V & \to & V\\
			& (z,w_2,\ldots,w_n) & \mapsto & (\tilde{f}(z,w_2,\ldots,w_n),w_2,\ldots,w_n)
		\end{matrix}
	\end{equation}
	is well-defined, and satisfies $\tilde{F} = \Phi_{\ell} \circ f \circ \Phi_{\ell}^{-1}$ in its domain of definition.

	\section{Existence of Fatou coordinates}\label{sect:Fatou}
	
	Let $\epsilon$, $\theta$ and $r$ be small enough so that lemmas \ref{lem:V} and \ref{lem:Vestimates} hold in $V = V (\epsilon, \theta, r)$. The goal of this section is to prove the following result:
	
	\begin{prop}\label{prop:Main}
		There exist an open set $W \in \CC^n$ and a biholomorphism $\B : V \to W$ such that $\B$ conjugates holomorphically the map $\tilde{F}$ (see equation \eqref{eq:TildeF}) to $(z, \w) \mapsto (z+1, \w)$ and
		$$
		\bigcup_{j \in \NN} T_j (W) = \CC \times (\CC^*)^{n-\bar{m}-1} \times \CC^{\bar{m}}.
		$$
	\end{prop}

	To prove this result, consider first $\w \in (\CC^*)^{n - \bar{m} -1} \times \CC^{\bar{m}}$. Since $\tilde{F} (z, \w) = (\tilde{f} (z, \w), \w)$, we know that the set $V_{\w} = \{z \in \CC, (z, \w) \in V\}$ is stable under $\tilde{f}_{\w} : z \mapsto \tilde{f}(z, \w)$ and we want to conjugate $\tilde{f}_{\w}$ to $z \mapsto z+1$. 
	
	From the definition of $V$ we have:
	$$
	V_{\w} = \{z \in \CC, \abs{z} > R_{\w}, \abs{\arg z}< \theta\} \quad \text{with} \quad R_{\w} = \max \left\{ \epsilon^{-1}, (r^{-1} \abs{\w}^{\mathcal{N}_i})^{\frac{1}{-\gamma/d - \Ree a_i}}\right\}.$$
	Notice that $V_{\w}$ is star-shaped.
	
	\begin{lem}\label{ref:beta}
		Let $p \in V_{\w}$ be big enough so that for all $z \in V_{\w}$, the euclidean segment $[z, p]$ is in $V_{\w}$.
		
		Then the sequence of functions
		$$
		\beta_j (z) = \tilde{f}_{\w}^{\circ j} (z) - \tilde{f}_{\w}^{\circ j} (p), \: j \in \NN.
		$$
		converges uniformly on $V_{\w}$ to a holomorphic map $\beta_{\w}$ that conjugates $\tilde{f}_{\w}$ to $z \mapsto z+1$.
	\end{lem}
	
	\begin{proof}
		Since $\abs{\tilde{f}_{\w}^{\circ j} (p)} > \abs{p}$ for all $j \in \NN$, $\tilde{f}_{\w}^{\circ j} (p)$ also satisfies that for all $z \in V_{\w}$, the euclidean segment $[z, \tilde{f}_{\w}^{\circ j} (p)]$ is in $V_{\w}$.
		
		Moreover as $\tilde{f}_{\w} (z) = z + 1 + \tilde{h} (z, \w)$ and $\abs{\frac{\partial \tilde{h}}{\partial z} (z, \w)} < K' \abs{z}^{-1-\gamma/d}$, we get from the mean value inequality:
		$$
		\abs{\frac{\tilde{f}_{\w} (z_1) - \tilde{f}_{\w} (z_2)}{z_1 - z_2} - 1} = \abs{\frac{\tilde{h}_{\w} (z_1) - \tilde{h}_{\w} (z_2)}{z_1 - z_2}} < \max_{z \in [z_1, z_2]} \frac{K'}{\abs{z}^{1 + \gamma/d}}.
		$$
		
		However we also know that since the opening of $V_{\w}$ is $2 \theta$, we have for all $z_1, z_2 \in V_{\w}$:
		$$
		\min_{z \in [z_1, z_2]} \abs{z} \geq \cos \theta \min \{\abs{z_1}, \abs{z_2}\}
		$$
		hence:
		$$
		\abs{\frac{\tilde{f}_{\w} (z_1) - \tilde{f}_{\w} (z_2)}{z_1 - z_2} - 1} \geq \frac{K'}{(\cos \theta \min \{\abs{z_1}, \abs{z_2}\})^{1 + \gamma/d}}.
		$$
		
		When we apply this formula to $z_1 = \tilde{f}_{\w} (z)$ and $z_2 = \tilde{f}_{\w} (p)$, we obtain:
		$$
		\abs{\frac{\beta_{j+1} (z)}{\beta_j (z)} - 1} = \abs{\frac{\tilde{f}_\w^{\circ j+1} (z) - \tilde{f}_\w^{\circ j+1} (p)}{\tilde{f}_\w^{\circ j} (z) - \tilde{f}_\w^{\circ j} (p)} - 1} \leq \frac{K'}{(\cos \theta \min \{\abs{\tilde{f}_\w^{\circ j} (z)}, \abs{\tilde{f}_\w^{\circ j} (p)}\})^{1 + \gamma/d}}
		$$
		for all $z \in V_{\w}$ and $j \in \NN$.
		
		As we have $\abs{\tilde{f}_{\w}^{\circ j} (z)} > j/2$ and $\abs{\tilde{f}_{\w}^{\circ j} (p)} > j/2$, it also holds that:
		$$
		\abs{\frac{\beta_{j+1} (z)}{\beta_j (z)} - 1} > K' (\frac{2}{\cos \theta})^{1+ \gamma/d} j^{-1 - \gamma/d}
		$$
		for all $z \in V_{\w}$ and $j \in \NN$. The product $\prod_{j \geq 0} \frac{\beta_{j+1} (z)}{\beta_j (z)}$ is therefore uniformly convergent in $V_{\w}$ and $(\beta_j)_j$ converges to a holomorphic function $\beta_{\w} \in \mathcal{O} (V_{\w})$.
		
		Since $\tilde{f}_\w (z) = z + 1 + \tilde{h}(z, \w)$, we can use equation \ref{eq:mh} to see that
		$$
		\abs{\tilde{f}_\w^{\circ j+1} (p) - \tilde{f}_\w^{\circ j} (p) - 1} = \abs{\tilde{h} (\tilde{f}_\w^{\circ j} (p), \w)} \geq \frac{K}{\abs{\tilde{f}_\w^{\circ j} (p)}^{\gamma/d}} \to_{j \to \infty} 0,
		$$
		so $\beta_j (\tilde{f}_\w (z)) = \beta_{j+1} (z) + \tilde{f}_\w^{\circ j+1} (p) - \tilde{f}_\w^{\circ j} (p)$ and when $j$ goes to $+\infty$ we get $\beta_\w (\tilde{f}_\w (z)) = \beta_\w (z) + 1$.
		
		Finally as $\beta_j$ is injective for all $j$ and the limit is not a constant, we see that $\beta_\w$ is injective.
	\end{proof}
	
	\begin{definition}
		Define $t_j : \CC \to \CC$ by $t_j (z) = z-j$ and $T_j : \CC^n \to \CC^n$ by $T_j (z_1, \ldots, z_n) = (z_1 - j, z_2, \ldots, z_n)$.
	\end{definition}
	
	\begin{lem}\label{lem:union}
		For all $\w \in \CC^{n-1}$, we have:
		$$
		\bigcup_{j \in \NN} t_j \left( \beta_\w (V_\w) \right) = \CC.
		$$ 
	\end{lem}
	
	\begin{proof}
		Let us first show that $\lim_{z \to \infty} \frac{\beta_\w (z)}{z} = 1$. Since $\beta_j$ converges uniformly to $\beta_\w$ when $j$ goes to infinity, there exists $l \in \NN$ such that $\abs{\beta_\w - \beta_l}$ is bounded, thus
		$$
		\abs{\beta_\w - \tilde{f}_\w^{\circ l}} \leq \abs{\beta_\w - \beta_l} + \abs{\beta_l -  \tilde{f}_\w^{\circ l}}
		$$
		is also bounded.
		
		As $\tilde{f}_\w (z) = z + 1 +\tilde{h} (z, \w)$ and $\abs{\tilde{h} (z, \w)} < K \abs{z}^{- \gamma/d}$, we have:
		$$
		\lim_{z \to \infty} \frac{\beta_\w (z)}{z} = \lim_{z \to \infty} \frac{\tilde{f}_\w^{\circ l} (z)}{z} = 1.
		$$
		
		Let $\zeta \in \CC$. In order to prove the lemma, we will show that $\zeta_j = \zeta + j$ is in $\beta_\w (V_\w)$ for $j$ big enough. Let $\rho < \sin \theta$, for $j$ big enough the disc $D_j$ of center $\zeta_j$ and of radius $r_j = \rho \abs{\zeta_j}$ is contained in $V_\w$.
		
		Using Rouché's theorem, if we have that for all $z \in \partial D_j$,
		$$
		\abs{\beta_\w (z) - z} < r_j = \abs{z - \zeta_j},
		$$
		then there exists a unique $z \in D_j$ such that $\beta_\w (z) = \zeta_j$ meaning that $\zeta_j \in \beta_\w (V_\w)$.
		
		Moreover as $\lim_{z \to \infty} \frac{\beta_\w (z)}{z} = 1$, for all $\tau > 0$ we have that for $z$ big enough, $\abs{\beta_\w (z) - z} < \tau \abs{z}$.
		
		Let us pick $\tau$ such that $\tau (1 + \rho) < \rho$, then
		$$
		\abs{\beta_\w (z) - z} < \tau \abs{z} \leq \tau (\abs{\zeta_j} + r_j) \leq \tau (1 + \rho) \abs{\zeta_j} \leq \rho \abs{\zeta_j} = r_j
		$$
		which concludes the proof.
		
	\end{proof}
	
	We are ready to prove the main result of this section:
	
	\begin{proof}[{Proof of Proposition \ref{prop:Main}}]
		The map $\beta_\w$ depends on the base point $p \in V_\w$, but its derivative does not:
		$$
		\beta_\w'(z) = \lim_{j \to \infty} (\tilde{f}_\w^{\circ j}) '(z).
		$$
		
		If $\w'$ is close to $\w$, then we can pick the same base point $p$ for $\beta_\w$ and $\beta_{\w'}$ and then $\beta_\w$ depends holomorphically on $\w$. We can thus find an open covering $\mathcal{U}$ of $(\CC^*)^{n-\bar{m}-1} \times \CC^{\bar{m}}$ and for all $U \in \mathcal{U}$, a holomorphic map $\beta_U (z, \w)$ for $\w \in U$ and $z \in V_\w$ such that $z \mapsto \beta_U (z, \w)$ is bijective for all $\w$ and $\beta_U (\tilde{f} (z, \w), \w) = \beta_U (z, \w) + 1$.
		
		Moreover, the derivative does not depend on $U \in \mathcal{U}$, as for all $\w \in U_1 \cap U_2$ and $z \in V_\w$ we have:
		$$
		\frac{\partial \beta_{U_1}}{\partial z} = \frac{\partial \beta_{U_2}}{\partial z}.
		$$
		Therefore if $U_1 \cap U_2 \neq \emptyset$, there exists $g_{U_1 U_2} \in \mathcal{O} (U_1 \cap U_2)$ such that
		$$
		\beta_{U_2} (z, \w) - \beta_{U_1} (z, \w) = g_{U_1 U_2} (\w)
		$$
		for $\w \in U_1 \cap U_2$ and $z \in V_\w$. These functions satisfy $g_{U_1 U_2} + g_{U_2 U_3} + g_{U_3 U_1} = 0$ on $U_1 \cap U_2 \cap U_3$ for all $U_1, U_2, U_3 \in \mathcal{U}$.
		
		Since the first Cousin problem can be solved on Stein manifolds \cite{Forstneric},
		there exists $g_U \in \mathcal{O} (U)$ such that $g_{U_1 U_2} = g_{U_1} - g_{U_2}$ on $W_{U_1} \cap W_{U_2}$, and moreover:
		$$
		\beta_{U_2} (z, \w) + g_{U_2} (z, \w) = \beta_{U_1} (z, \w) + g_{U_1} (z, \w)
		$$
		for $\w \in U_1 \cap U_2$ and $z \in V_\w$.
		
		We can thus define a global function $\beta \in \mathcal{O} (V)$ by $\beta (z, \w) = \beta_U (z, \w) + g_U (\w)$ on $U$, and for all $\w \in (\CC^*)^{n-\bar{m}-1} \times \CC^{\bar{m}}$, $z \mapsto \beta (z, \w)$ is univalent and $\beta (\tilde{f} (z, \w), \w) = \beta (z, \w) + 1$.
		
		Then it is clear that the map $\B (z, \w) = (\beta(z,\w), \w)$ is bijective and that $\B \circ \tilde{F} (z, \w) = \B (z, \w) + (1, 0, \ldots, 0)$ for every $(z, \w) \in V$.
		
		Let us now prove that $W = \B (V)$ satisfies:
		$$
		\bigcup_{j \in \NN} T_j (W) = \CC \times (\CC^*)^{n-\bar{m}-1} \times \CC^{\bar{m}}.
		$$
		Consider a point $(z_0, \w_0) \in \CC \times (\CC^*)^{n-\bar{m}-1} \times \CC^{\bar{m}}$, if $\w_0 \in W_U$ we have $\beta (z_0, \w_0) = \beta_U (z_0, \w_0) + g_U (\w_0) = \beta_{\w_0} (z_0) + g_U (\w_0)$ for all $z \in V_{\w_0}$. By lemma \ref{lem:union} there exists $z \in V_{\w_0}$ and $j \in \NN$ such that $z_0 - g_U (\w_0) + j = \beta_{\w_0} (z)$ and then
		$$
		\B (z, \w_0) = (\beta (z, \w_0), \w_0) = (z_0 + j, \w_0)
		$$
		so $(z_0, \w_0) \in T_j (W)$.
	\end{proof}
	
	\section{The flower theorem}\label{sect:Proof}
	
	In this section we complete the proof of theorem \ref{thm:conv}. We have already built open sets $(U_\ell)_{\ell \in \{1, \ldots, n\}}$ and biholomorphisms $\B \circ \Phi_\ell : U_\ell \to W$ that conjugate $f$ to $(z, \w) \mapsto (z+1, \w)$ when $f$ satisfies $\Ree \left(\frac{a_i}{\aM} \right) > 0$ for all $i \in \{1, \ldots, n\}$ .
	
	We can do the same with $f^{-1}$, however the reunion of these open sets with the fixed point set do not cover a neighbourhood of the origin. We thus need to enlarge them while preserving the properties of convergence and conjugation.
	
	Let $\epsilon, \delta, r > 0$ and $\theta \in ]0, \pi/2[$ such that everything holds.
	
	Let us consider the domain $\tilde{S}_{-1} (\epsilon, \theta)$ introduced in definition \ref{def:S}, and let $\tilde{S}_\ell$, for $\ell \in \{0, \ldots, d-1\}$ be its connected components bisected by $e^{\frac{2 i \pi \ell}{d}} \RR^+$. It is a sectorial domain of opening $\frac{\pi + 2 \theta}{d}$.
	
	Let $0 < \delta' \leq \delta$, for $\ell \in \{0, \ldots, d-1\}$ we define
	$$
	\tilde{D}_\ell (\delta') = \{\x \in \CC^{n}, \x^m \in \tilde{S}_\ell, \abs{x_i} < \delta' \forall i\}.
	$$
	It is connected as $\gcd (m) = 1$.
	
	\begin{lem}\label{lem:D}
		If $\delta'$ is small enough, for all $p \in \tilde{D}_\ell (\delta')$, there exists $j \in \NN$ such that $f^{\circ j} (p) \in \Phi_\ell^{-1} (V)$.
	\end{lem}
	
	\begin{proof}
		Without loss of generality we can assume $\ell = 0$.
		
		We claim that for any $\x \in \tilde{D}_0 (\delta')$, there exists $j \in \NN$ such that $f^{\circ j} (\x) \in D_0 (\epsilon, \theta, \delta)$. Then this claim together with remarks \ref{rem:D} and \ref{rem:V} prove the lemma.
		
		We use the estimates of Leau-Fatou theorem on equation \eqref{eq:xm} to get two constants $c \geq 1$ and $C>0$ depending only on $\epsilon$ and $\theta$ such that if $\x \in \CC^n$ satisfies $\x^m \in \tilde{S}_0 (\epsilon, \theta)$ and $\abs{f^{\circ j}_i (\x)} < \delta$ for all $i \in \{1, \ldots, n\}$ and $j \in \NN$, then we have that $f^{\circ j} (\x)^m \in \tilde{S}_0$ for all $j$, $\abs{f^{\circ j} (\x)^m}^d \leq c \abs{\x^m}^d$ for all $j$, and $\abs{\arg (f^{\circ j} (\x)^m)} < \theta/d$ for $j \geq \frac{C}{\abs{\x^m}^d}$.
		
		We can find $\rho > 0$ such that for all $i$,
		$$
		\abs{f_i (\x)} = \abs{x_i (1 + a_i \x^m + o (\x^m))} \leq \abs{x_i} (1 + \rho \abs{\x^m}).
		$$
		Then set $K = \sup_{t \in ]0, \epsilon[} (1 + \rho c t^d)^{\frac{C}{t^d + 1}}$ and $\delta' < \frac{\delta}{K}$. Let $\x \in \tilde{D}_0 (\delta')$ and $j_0 = \lceil \frac{C}{\abs{\x^m}^d} \rceil$. Then for all $j \geq j_0$, we have:
		$$
		\abs{f^{\circ j}_i (\x)} \leq \abs{x_i} \prod_{l=0}^{j-1} (1 + \rho \abs{f^{\circ l} (\x)^m}^d) \leq \abs{x_i} (1 + \rho c \abs{\x^m}^d)^{j_0}
		$$
		so $\abs{f^{\circ j}_i (\x)} \leq \delta' K < \delta$ for all $j \leq j_0$.
		
		As we also have that $\abs{\arg (f^{\circ j} (\x)^m)} < \theta/d$ for $j \geq j_0$, it follows that $f^{\circ j_0} (\x) \in D_0 (\epsilon, \theta, \delta)$.
		
	\end{proof}
	
	\begin{proof}[End of the proof of theorem \ref{thm:conv}]
		
		Let $\delta' > 0$ such that lemma \ref{lem:D} holds. For each $\ell \in \{0, \ldots, d-1\}$ we define
		$$
		\Omega_\ell^+ = \bigcup_{j \geq 0} f^{\circ j} (\tilde{D}_\ell (\delta')).
		$$
		Then $\Omega_\ell^+$ is invariant by $f$, is connected and is attracted to $U_\ell$ by lemma \ref{lem:D}.
		
		We have $\tilde{D}_\ell (\delta') \subset \Omega_\ell^+$ and from the proof of lemma \ref{lem:D} we also get $\Omega_\ell^+ \subset \tilde{D}_\ell (K \delta')$.
		
		As $f^{-1}$ can be written as:
		$$
		f^{-1}_i (\x) = x_i (1 + \x^m (-a_i + o(1)))
		$$
		we can apply the same methods to $f^{-1}$ instead of $f$ and build connected open domains $\Omega_0^-, \ldots, \Omega_{d-1}^-$ defined by
		$$
		\Omega_\ell^- = \bigcup_{j \geq 0} f^{\circ - j} (\tilde{D}_\ell^- (\delta'))
		$$
		with
		$\tilde{D}_\ell^- (\delta') = \{\x \in \CC^{n}, \x^m \in \tilde{S}_\ell^-, \abs{x_i} < \delta' \: \forall i\}$ where $\tilde{S}_\ell^-$ are the connected components of $\tilde{S}_1$ as introduced in definition \ref{def:S}. Then in each $\Omega_\ell^-$, we have $f^{\circ - j} \to 0$ and $\tilde{D}_\ell^- (\delta') \subset \Omega_\ell^- \subset \tilde{D}_\ell^- (K \delta')$.
		
		Since the opening of the connected components of $\tilde{S}_{-1} (\epsilon, \theta)$ and of $\tilde{S}_1 (\epsilon, \theta)$ is bigger than $\pi/d$, we have that the domains $\Omega_0^+, \ldots, \Omega_{d-1}^+, \Omega_0^-, \ldots, \Omega_{d-1}^-$ together with the fixed point set $\{\x^m = 0\}$ cover a neighbourhood of the origin.
		
		For each $\ell$, the map $\phi_\ell^+ = \B \circ \Phi_\ell : \Phi_\ell^{-1} (V) \to \CC^n$ is a biholomorphism that conjugates $f$ with $(z, \w) \mapsto (z+1, \w)$. It can naturally be extended to $\Omega_\ell^+$ as for every $\x \in \Omega_\ell^+$, there exists $j \in \NN$ such that $f^{\circ j} (\x) \in \Phi_\ell^{-1} (V)$. We then define $\phi_\ell^+ (\x) = \phi_\ell^+ (f^{\circ j} (\x)) - (j, 0, \ldots, 0)$.
		
		Moreover we can see that $\phi_\ell^+ (\Omega_\ell^+) \subset \CC \times (\CC^*)^{n-\bar{m}-1} \times \CC^{\bar{m}}$.
		
	\end{proof}

	\section{Proof of theorem \ref{thm:div}}\label{sect:div}
	
	Consider $f : (\CC^n, 0) \to (\CC^n, 0)$ of the form \eqref{eq:f}, after the change of coordinates done in the first section we can assume $\aM = -1$. Then the condition in theorem \ref{thm:div} translates into $\Ree (a_i) > 0$ for some $i \in \{1, \ldots, n\}$.
	
	Given $\delta > 0$, we define $U$ as the polydisc:
	$$
	U (\delta) = \{ \x \in \CC^n : \abs{x_i} < \delta \: \forall i\}.
	$$
	
	\begin{prop}
		For all $\x \in U$ outside of the fixed point set $\x^M = 0$, there exists $j > 0$ such that $f^{\circ j} (\x) \notin U$, and there exists $j' > 0$ such that  $f^{\circ -j'} (\x) \notin U$.
	\end{prop}
	
	\begin{proof}
		We only prove the first assertion, as the second one follows from studying $f^{-1}$ in the same way.
		
		Let $\x$ be a point in $U$ such that $\x^M \neq 0$, then by Leau-Fatou theorem applied to $\x^M$, either the orbit of $\x$ escapes $U$ or $f^{\circ j}(\x)^M$ tends to zero. In the second case, we then also know that it will satisfy :
		$$
		j f^{\circ j}(\x)^M \to 1
		$$
		which implies that the sum $\sum_{j \geq 0} \abs{f^{\circ j}(\x)^M}$ diverges and that $\abs{\arg (f^{\circ j} (\x)^M))} \to 0$.
		
		Moreover, there exists $j_0 \geq 0$ and $\nu >0$ such that for all $j \geq j_0$,
		\begin{align*}
			\abs{f^{\circ j+1}_i(\x)} &= \abs{f^{\circ j}_i(\x)} \abs{1 + a_i \x^M + o (\x^M)}\\
			&\geq \abs{f^{\circ j}_i(\x)} \left( 1+ \nu \abs{\x^M} \right)
		\end{align*}
		as $\Ree (a_i) > 0$.
		
		So
		$$
		\abs{f^{\circ j}_i(\x)} \geq \abs{f^{\circ j_0}_i(\x)} \prod_{k = j_0}^{j-1} \left(1 + \nu \abs{f^{\circ k}(\x)^M} \right)
		$$	
		and the product tends to $\infty$ so there exists $j$ such that the right hand term is bigger than $\delta$ which means $f^{\circ j} (\x) \neq U$.
	\end{proof}
	
	\bibliographystyle{alpha}
	\bibliography{../../bibliographie}

\end{document}